\documentclass[12pt]{amsart}
\usepackage{amssymb}
\usepackage[mathscr]{eucal}
\usepackage{parskip}
\usepackage{xcolor}
\usepackage{tikz}
\usetikzlibrary{calc}
\usepackage{tikz-cd}
\usepackage{hyperref}
\numberwithin{equation}{section}
\usepackage[margin=2.9cm]{geometry}
\usepackage{booktabs}
\usepackage{enumerate}
\usepackage{enumitem}
\usepackage{bm}
\usepackage[all]{xy}
\usepackage{soul}

\theoremstyle{plain}
\newtheorem{thm}{Theorem}[section]
\newtheorem{lem}[thm]{Lemma}

\newtheorem{prop}[thm]{Proposition}

 \theoremstyle{definition}
\newtheorem{defn}[thm]{Definition}
\newtheorem{rem}[thm]{Remark}

\newtheorem{notn}[thm]{Notation}

\newcommand{\mb}[1]{\mathbb{#1}}
\newcommand{\mc}[1]{\mathcal{#1}}

\newcommand{\mr}[1]{\mathrm{#1}}

\newcommand{\vphi}{\varphi}
\newcommand{\id}{\operatorname{id}}
\newcommand{\Spec}{\operatorname{Spec}}
\newcommand{\GW}{\mathrm{GW}}
\newcommand{\Tr}{\mathrm{Tr}}
\newcommand{\Hom}{\mathcal{H}om}
\newcommand{\ind}{\mathrm{ind}}

\newcommand{\PGr}{\mathbb{G}\mathrm{r}}

\newcommand{\rank}{\operatorname{rank}}
\newcommand{\Sym}{\mathrm{Sym}}
\newcommand{\spn}{\operatorname{span}}
\newcommand{\Jac}{\operatorname{Jac}}
\renewcommand{\d}{\mathrm{d}}
\newcommand{\T}{\mathrm{T}}
\newcommand{\Wr}{\operatorname{Wr}}

\newcommand{\II}{\sbox0{II}\dimen0=\dimexpr\wd0+1pt\relax
  \makebox[\dimen0]{\rlap{\vrule width\dimen0 height 0.07ex depth 0.07ex}%
    \rlap{\vrule width\dimen0 height\dimexpr\ht0+0.07ex\relax 
            depth\dimexpr-\ht0+0.07ex\relax}%
    \kern.5pt \textnormal{II}\kern.5pt}}
\newcommand{\III}{\sbox0{III}\dimen0=\dimexpr\wd0+1pt\relax
  \makebox[\dimen0]{\rlap{\vrule width\dimen0 height 0.07ex depth 0.07ex}%
    \rlap{\vrule width\dimen0 height\dimexpr\ht0+0.07ex\relax 
            depth\dimexpr-\ht0+0.07ex\relax}%
    \kern.5pt \textnormal{III}\kern.5pt}}

\makeatletter
\@namedef{subjclassname@2020}{\textup{2020} Mathematics Subject Classification}
\makeatother

\begin{document}
\title{Quadratic counts of highly tangent lines to hypersurfaces}

\author[McKean]{Stephen McKean}
\address{Department of Mathematics \\ Brigham Young University} 
\email{mckean@math.byu.edu}
\urladdr{shmckean.github.io}

\author[Muratore]{Giosu{\`e} Muratore}
\address{CEMS.UL (University of Lisbon), and  
COPELABS/DEISI (Lus\'ofona University)}
\email{muratore.g.e@gmail.com}
\urladdr{sites.google.com/view/giosue-muratore}

\author[Ong]{Wern Juin Gabriel Ong}
\address{Department of Mathematics \\ Universit\"at Bonn}
\email{wgabrielong@uni-bonn.de}
\urladdr{wgabrielong.github.io}


\begin{abstract}
We give two geometric interpretations for the local type of a line that is highly tangent to a hypersurface in a single point. One interpretation is phrased in terms of the Wronski map, while the other interpretation relates to the fundamental forms of the hypersurface. These local types are the local contributions of a quadratic form-valued Euler number that depends on a choice of orientation.
\end{abstract}

\maketitle

\section{Introduction}
Over an algebraically closed field, the number of roots of a polynomial, when counted with multiplicity, is given by its degree. In contrast, the number of real roots is not determined by the degree. In order to give a count of real roots that depends only on the degree of the polynomial, the roots must be counted with a \emph{sign} rather than a multiplicity (see Figure~\ref{fig:polynomials}).

\begin{figure}[h]
\begin{tikzpicture}[scale=.75]
  \draw[->] (-2, 0) -- (2, 0);
  \draw[->] (0, -2) -- (0, 2);
  \draw[scale=0.5, domain=-2.25:2.25, smooth, variable=\x, blue, very thick] plot ({\x}, {\x*\x-2});
  \node[below,red] at (-1,0) {$-$};
  \node[below,red] at (1,0) {$+$};
\end{tikzpicture}\quad
\begin{tikzpicture}[scale=.75]
  \draw[->] (-2, 0) -- (2, 0);
  \draw[->] (0, -2) -- (0, 2);
  \draw[scale=0.5, domain=-1.75:1.75, smooth, variable=\x, blue, very thick] plot ({\x}, {2*\x*(\x-1.5)*(\x+1.5)});
  \node[above,red] at (-1,0) {$+$};
  \node[above,red] at (.25,0) {$-$};
  \node[below,red] at (1,0) {$+$};
\end{tikzpicture}
\quad
\begin{tikzpicture}[scale=.75]
  \draw[->] (-2, 0) -- (2, 0);
  \draw[->] (0, -2) -- (0, 2);
  \draw[scale=0.5, domain=-1.75:1.75, smooth, variable=\x, blue, very thick] plot ({\x}, {-2*\x*(\x-1.5)*(\x+1.5)});
  \node[below,red] at (-1,0) {$-$};
  \node[above,red] at (-.25,0) {$+$};
  \node[above,red] at (1,0) {$-$};
\end{tikzpicture}
\quad
\begin{tikzpicture}[scale=.75]
  \draw[->] (-2, 0) -- (2, 0);
  \draw[->] (0, -2) -- (0, 2);
  \draw[scale=0.5, domain=-1.75:1.75, smooth, variable=\x, blue, very thick] plot ({\x}, {1.3*\x*\x*(\x-1.5)*(\x+1.5)});
  \node[below,red] at (-1,0) {$-$};
  \node[above,red] at (-.2,0) {$0$};
  \node[below,red] at (1,0) {$+$};
\end{tikzpicture}
    \caption{Signed counts of real roots}\label{fig:polynomials}
\end{figure}
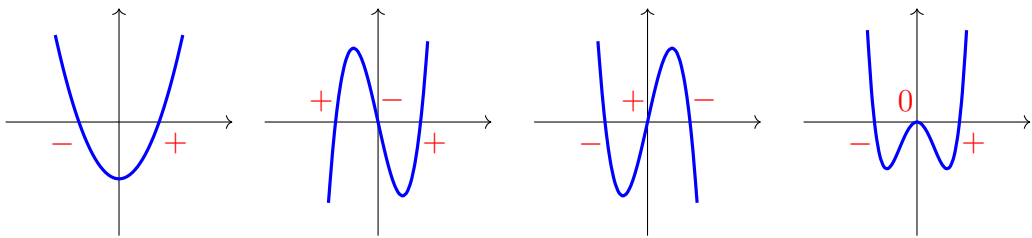

For polynomials of even degree, this signed count of real roots is always 0. For polynomials of odd degree, the signed count is $\pm 1$. This sign ambiguity can be interpreted as an artifact of non-orientability: the roots of a polynomial correspond to the intersection of two plane curves, and B\'ezout's theorem is not an orientable problem over non-closed fields \cite{MR4211099}.

By \emph{highly tangent line} we mean a line meeting a hypersurface in $\mb{P}^n$ at a point with contact order at least $2n-1$. In this article, we will investigate the geometric weight with which to count highly tangent lines to the hypersurface. Analogous to the case of real roots of odd polynomials, this enumerative problem is not orientable in many cases of interest, so the total count of highly tangent lines depends on the choice of hypersurface. Despite the lack of invariant total count, we find that this counting weight admits two interesting descriptions.

\begin{thm}[Wronskian interpretation]\label{thm:wronskian}
    Let $k$ be a field. Let $X=\mb{V}(F)$ be a smooth, general hypersurface in $\mb{P}^n_k$ of degree $d\geq 2n-1$. Let $\Phi$ be the flag variety of pointed lines in $\mb{P}^n_k$, and let $\beta\colon \Phi\to\mb{P}^n_k$ send a pointed line to its marked point.
    
    Let $L$ be a line such that the intersection multiplicity of $L$ and $X$ at $p$ is $2n-1$. Then the local index associated to $L$ is given by the Wronskian determinant (with respect to the local parameter of $L$ at $p$) of the gradient $\nabla\beta^*F$.
\end{thm}

\begin{thm}[Fundamental form interpretation]\label{thm:fundamental form}
Let $X$ be a smooth, general, projective plane curve, and assume $\mathrm{char}(k)\neq2,3$. Let $\inf(X)$ denote the locus of inflectional points on $X$. Let $\II$ and $\III$ denote the second and third fundamental forms of $X$, respectively. For each $p\in\inf(X)$, we have
\[\ind_p\II=\Tr_{k(p)/k}\langle 3\cdot\III(p)\rangle.\]
Moreover, if $X$ is of even degree $d$ and admits a theta characteristic, then
\[\sum_{p\in\inf(X)}\Tr_{k(p)/k}\langle\III(p)\rangle=\frac{3d(d-2)}{2}\mb{H},\]
where $\mb{H}:=\langle1\rangle+\langle-1\rangle$.
\end{thm}

We will make these theorems more precise in Sections~\ref{sec:wronskian} and~\ref{sec:fundamental form}, respectively.

Our work fits into the \emph{enriched enumerative geometry} program, which consists of quadratic form-valued counts of geometric objects. These counts are usually valid over more general base fields than just algebraically closed or real fields. Quadratic forms arise as a result of the tools necessary to work over arbitrary fields --- motivic homotopy theory provides suitable substitutes for Euler numbers and the local Brouwer degree, but these invariants are valued in the Grothendieck--Witt ring of the base field rather than in $\mb{Z}$.

\subsection{Outline}
The enumerative problem of counting highly tangent (projective) linear spaces to hypersurfaces can be formulated in terms of a bundle of principal parts over a flag variety of pointed linear spaces (see e.g.~\cite[\S 11.1]{3264}). In Section~\ref{sec:parameter space}, we will recall the construction of this flag variety, which serves as the parameter space of objects that we will count. We will also discuss standard coordinates on the flag variety. In Section~\ref{sec:bundle}, we will recall the bundle of principal parts, as well as convenient local trivializations of this bundle. We will also make a few remarks about Euler numbers in this section.

As previously mentioned, the problem we are treating is not orientable. To resolve this, we need to choose an orienting divisor in our parameter space, which we do in Section~\ref{sec:orienting divisors}. We then prove Theorem~\ref{thm:wronskian} in Section~\ref{sec:wronskian}.

An alternative formulation of the problem of counting highly tangent lines involves fundamental forms. An advantage of this formulation is that the problem is relatively orientable in cases where the bundle of relative principal parts over the flag variety is not relatively orientable. We will set up the necessary background and prove Theorem~\ref{thm:fundamental form} in Section~\ref{sec:fundamental form}.

Note that most computations of local indices in the enriched enumerative geometry literature involve a parameter space whose Nisnevich coordinates are actually isomorphisms, because the parameter space is covered by affine spaces. In contrast, our computation of the relevant Euler class in the second approach uses non-trivial Nisnevich coordinates. So our second approach adds a bit of novelty to the literature.

\subsection*{Data management statement}
The data that support the findings of this study are openly available in \texttt{highly\textunderscore{}tangent} at \url{https://github.com/wgabrielong/highly_tangent/tree/main}.

\subsection*{Conflict of interest statement}
The authors declare no conflict of interest related to this study.

\subsection*{Acknowledgements}
We thank Ethan Cotterill, Steven Kleiman, Marc Levine, Wenbo Niu, Sabrina Pauli, and Felipe Voloch for useful discussions. SM received support from an NSF MSPRF grant (DMS-2202825). GM is a member of GNSAGA (INdAM), and is supported by FCT - Funda\c{c}\~{a}o para a Ci\^{e}ncia e a Tecnologia, under the project: UID/04561/2025. WO thanks DJ Merril for help with implementing several of the computations in \cite{GabRep}.

\section{The parameter space}\label{sec:parameter space}
Consider the Grassmannian $\PGr(r,n)$ of projective $r$-planes in $\mb{P}^n$, which can also be thought of as the space of affine $(r+1)$-planes through the origin in $\mb{A}^{n+1}$. Our parameter space is a flag variety that naturally occurs as an incidence variety over $\PGr(r,n)$.

\begin{defn}\label{def:Phi}
    Define the flag variety of pointed planes as the incidence variety
    \[\Phi_{r,n}:=\{(H,p)\in\PGr(r,n)\times\mb{P}^n:p\in H\}.\]
\end{defn}

Note that the variety $\Phi_{r,n}$ is isomorphic to the projective bundle $\pi\colon \mb{P}\mc{S}\to\mb{G}(r,n)$, where $\mc{S}\to\mb{G}(r,n)$ is the tautological bundle over the affine Grassmannian. We will also use $\pi$ when discussing the pair of projections
\[\begin{tikzcd}[sep=small]
    & \Phi_{r,n}\arrow[dl,"\pi"']\arrow[dr,"\beta"] &\\
    \PGr(r,n) & & \mb{P}^n,
\end{tikzcd}\]
which should not cause confusion as $\Phi_{r,n}\cong\mb{P}\mc{S}$.

Before discussing coordinates for $\Phi_{r,n}$, we mention a few facts about $\Phi_{r,n}$ that we will need later. All of these facts are standard computations for projective bundles.

\begin{lem}\label{lem:dim Phi}
    We have $\dim\Phi_{r,n}=r(n-r)+n$.
\end{lem}
\begin{proof}
    Since $\Phi_{r,n}\cong\mb{P}\mc{S}$, we can compute its dimension as a projective bundle over $\mb{G}(r,n)$. This gives us
    \begin{align*}
        \dim\mb{P}\mc{S}&=\dim\mb{G}(r,n)+\rank\mc{S}-1\\
        &=(r+1)(n-r)+r+1-1\\
        &=r(n-r)+n.\qedhere
    \end{align*}
\end{proof}

We fix the following notation for the subsequent three lemmas.

\begin{notn}
    Let $\Phi:=\Phi_{r,n}$ and $G:=\mb{G}(r,n)$. Let $\Omega_{\Phi/G}$ denote the relative cotangent bundle of $\pi\colon \Phi\to G$. Let $\omega_X$ denote the canonical bundle of a scheme $X$.
\end{notn}

\begin{lem}\label{lem:canonical bundle Phi}
    We have $\omega_\Phi\cong\mc{O}_\Phi(-r-1)\otimes\pi^*\mc{O}_G(-n)$.
\end{lem}
\begin{proof}
    For any vector bundle $V\to X$ of rank $v$ with associated projective bundle $\pi\colon \mb{P}V\to X$, there is an isomorphism
    \[\omega_{\mb{P}V}\cong\mc{O}_{\mb{P}V}(-v)\otimes\pi^*\det V^\vee\otimes\pi^*\omega_X.\]
    The result now follows from $\det\mc{S}\cong\mc{O}_G(-1)$ and $\omega_G\cong\mc{O}_G(-n-1)$, which are standard facts about Grassmannians and their tautological bundles.
\end{proof}

\begin{lem}\label{lem:rel cotangent bundle Phi}
    We have $\rank\Omega_{\Phi/G}=r$.
\end{lem}
\begin{proof}
    Since $\Phi$ and $G$ are both smooth, we have $\rank\Omega_\Phi=\dim\Phi$ and $\rank\Omega_G=\dim{G}$. The rank of the relative cotangent bundle is thus
    \begin{align*}
        \rank\Omega_{\Phi/G}&=\rank\Omega_\Phi-\rank\Omega_G\\
        &=\dim\Phi-\dim G\\
        &=r(n-r)+n-(r+1)(n-r)\\
        &=r.\qedhere
    \end{align*}
\end{proof}

\begin{lem}\label{lem:det Omega}
    We have $\det\Omega_{\Phi/G}\cong\mc{O}_\Phi(-r-1)\otimes\pi^*\mc{O}_G(1)$.
\end{lem}
\begin{proof}
    By the dual relative Euler sequence
    \[0\to \pi^*\Omega_G\to\Omega_\Phi\to\Omega_{\Phi/G}\to 0,\]
    we have $\det\Omega_{\Phi/G}\cong\omega_\Phi\otimes\pi^*\omega_G^\vee$. Applying Lemma~\ref{lem:canonical bundle Phi}, we find
    \begin{align*}
        \det\Omega_{\Phi/G}&\cong\omega_\Phi\otimes\pi^*\omega_G^\vee\\
        &\cong\mc{O}_\Phi(-r-1)\otimes\pi^*\mc{O}_G(-n)\otimes\pi^*\mc{O}_G(n+1)\\
        &\cong\mc{O}_\Phi(-r-1)\otimes\pi^*\mc{O}_G(1).\qedhere
    \end{align*}
\end{proof}

\subsection{Coordinates}\label{sec:coordinates}
We need a suitable choice of coordinates on $\Phi_{r,n}$ in order to describe the local type of a highly tangent plane to a hypersurface. We will choose coordinates on $\Phi_{r,n}$ by exploiting its structure as a projective bundle. In short, we will use the standard coordinates on $\mb{G}(r,n)$ for the base and a twist of the standard coordinates on $\mb{P}^r$ for the fibers.

Let $\{e_1,\ldots,e_{n+1}\}$ denote the standard basis for $k^{n+1}$. Every $k$-rational point of $\mb{G}(r,n)$ can be obtained as the span of a column of the form
\begin{equation}\label{eq:column}
\begin{pmatrix}
    \tilde{x}_{1,1} & \cdots & \tilde{x}_{1,n+1}\\
    \tilde{x}_{2,1} & \cdots & \tilde{x}_{2,n+1}\\
    \vdots & \ddots & \vdots\\
    \tilde{x}_{r+1,1} & \cdots & \tilde{x}_{r+1,n+1}
\end{pmatrix}
\begin{pmatrix}
e_1 \\ e_2 \\ \vdots \\ e_{n+1}
\end{pmatrix}.\end{equation}
This description generalizes beyond the $k$-points of $\mb{G}(r,n)$ and gives natural coordinates for the space. The standard open cover of $\mb{G}(r,n)$ is given by taking non-vanishing $(r+1)\times(r+1)$-minors. 

More precisely, let $1\leq i_1<\ldots<i_{r+1}\leq n+1$, and denote $I=\{i_1,\ldots,i_{r+1}\}$. Let $\{c_1,\ldots,c_{n-r}\}=\{1,\ldots,n+1\}-I$, with $c_1<\ldots<c_{n-r}$. Let $U_I\subset\mb{G}(r,n)$ denote the open set of $(r+1)$-planes such that $\det((\tilde{x}_{t,i_s})_{t,s=1}^{r+1})\neq 0$ (in the parameterization given in Equation~\ref{eq:column}). By row reducing, all such $(r+1)$-planes in $U_I$ can be obtained by taking $(\tilde{x}_{t,i_s})_{t,s=1}^{r+1}$ to be the identity matrix. We thus obtain an isomorphism
\[\vphi_I\colon U_I\to\Spec k[x_{1,1},\ldots,x_{r+1,n-r}],\]
under which a point $(p_{1,1},\ldots,p_{r+1,n-r})\in\Spec k[x_{1,1},\ldots,x_{r+1,n-r}]$ corresponds to the span of the vectors
\[\tilde{e}_{i_s}:=e_{i_s}+\sum_{j=1}^{n-r}p_{s,j}e_{c_j}\]
for $s=1,\ldots,r+1$.

\begin{defn}
    The \emph{standard coordinates} for $\mb{G}(r,n)$ are the collection of local coordinates $\{(U_I,\vphi_I)\}_I$.
\end{defn}

Now let $\varpi\colon \mc{S}\to\mb{G}(r,n)$ denote the tautological bundle. If $H\in U_I$, then $H$ is the span of $\{\tilde{e}_{i_1},\ldots,\tilde{e}_{i_{r+1}}\}$ (for some $x_{1,1},\ldots,x_{r+1,n-r}$) as described above. If we denote the coordinates of $\mb{A}^{r+1}$ by $(y_1,\ldots,y_{r+1})$, then $\varpi^{-1}(H)\cong\mb{A}^{r+1}$ parameterizes vectors of the form
\begin{equation}\label{eq:v}
v=\sum_{s=1}^{r+1}y_s\tilde{e}_{i_s},
\end{equation}
as these are precisely the vectors belonging to $H$. This gives us local coordinates $\{(\varpi^{-1}(U_I),\vphi_I\times\id)\}_I$ on the total space of $\varpi\colon \mc{S}\to\mb{G}(r,n)$. To obtain local coordinates on $\Phi_{r,n}\cong\mb{P}\mc{S}$, we take a twist of the standard open cover of $\mb{P}^r$ on the fibers.

\begin{defn}\label{def:coordinates on Phi}
    Let $U_{I,\ell}\subset\mb{P}\mc{S}$ be the open set of pairs $(H,\mr{span}(v))$, where $H\in U_I$ and the coordinate of $\tilde{e}_{i_\ell}$ in $v$ (from Equation~\ref{eq:v}) is non-zero. Let $\vphi_{I,\ell}:=\vphi_I\times\psi_\ell$, where $\psi_\ell:\{[y_1:\cdots:y_{r+1}]:y_\ell\neq 0\}\to\mb{A}^r$ are the twisted affine charts on $\mb{P}^r$ defined in \cite[p.~638]{MR4211099}. The \emph{standard coordinates} for $\Phi_{r,n}\cong\mb{P}\mc{S}$ consist of the local coordinates $\{(U_{I,\ell},\vphi_{I,\ell})\}_{I,\ell}$.
\end{defn}

\subsection{The case of lines}
As we will restrict our attention to the $r=1$ case later in this article, we now make Definition~\ref{def:coordinates on Phi} a little more explicit in this case. Let
\[(H,\spn(v))=\left(\spn\left\{\sum_{i=1}^{n+1}\tilde{x}_{1,i}e_i,\sum_{i=1}^{n+1}\tilde{x}_{2,i}e_i\right\},\spn\left\{\sum_{j=1}^2 y_j\sum_{i=1}^{n+1}\tilde{x}_{j,i}e_i\right\}\right).\]
The map $\beta\colon \Phi_{r,n}\to\mb{P}^n$ is given by
\[(H,\spn(v))\mapsto\spn\left\{\sum_{i=1}^{n+1}(y_1\tilde{x}_{1,i}+y_2\tilde{x}_{2,i})e_i\right\},\]
and the Pl\"ucker map composed with $\pi\colon \Phi_{r,n}\to\mb{G}(1,n)$ is given by
\[(H,\spn(v))\longmapsto\spn\left\{\sum_{i=1}^n\sum_{j=i+1}^{n+1}(\tilde{x}_{1,i}\tilde{x}_{2,j}-\tilde{x}_{1,j}\tilde{x}_{2,i})e_i\wedge e_j\right\}.\]
The open subsets $U_{I,\ell}$ (where $\ell\in\{1,2\}$) parameterize pairs $(H,\spn(v))$ satisfying
\begin{align*}
    \tilde{x}_{1,i_1}\tilde{x}_{2,i_2}-\tilde{x}_{1,i_2}\tilde{x}_{2,i_1}&\neq 0,\\
    y_1\tilde{x}_{1,i_\ell}+y_2\tilde{x}_{2,i_\ell}&\neq 0.
\end{align*}
The coordinates $\vphi_{I,\ell}\colon U_{I,\ell}\to\mb{A}^{2(n-2)}\times\mb{A}^1$ are defined by
\[\vphi_{I,\ell}(H,\spn(v))=\left(\vphi_I(H),(-1)^{\ell-1}\cdot\frac{y_1\tilde{x}_{1,i_{\ell'}}+y_2\tilde{x}_{2,i_{\ell'}}}{y_1\tilde{x}_{1,i_\ell}+y_2\tilde{x}_{2,i_\ell}}\right),\]
where $\ell'$ is the remaining element of $\{1,2\}-\{\ell\}$.

Computing the transition functions of $\{(U_{I,\ell},\vphi_{I,\ell})\}_{I,\ell}$ is a standard computation, the details of which we will omit. For our purposes, it suffices to know that the determinant of the Jacobian matrix of $\vphi_{J,m}\circ\vphi_{I,\ell}^{-1}$ is given by
\begin{equation}\label{eq:jacobian of transition}
(-1)^{\ell+m}\left(\frac{\tilde{x}_{1,i_1}\tilde{x}_{2,i_2}-\tilde{x}_{1,i_2}\tilde{x}_{2,i_1}}{\tilde{x}_{1,j_1}\tilde{x}_{2,j_2}-\tilde{x}_{1,j_2}\tilde{x}_{2,j_1}}\right)^n\left(\frac{y_1\tilde{x}_{1,i_\ell}+y_2\tilde{x}_{2,i_\ell}}{y_1\tilde{x}_{1,j_m}+y_2\tilde{x}_{2,j_m}}\right)^2.
\end{equation}

\section{The bundle}\label{sec:bundle}
We will briefly revise the definition of 
the flag variety $\Phi_{r,n}$ parameterizing projective $r$-planes in $\mb{P}^n$ equipped with a point. We now need a vector bundle encoding the condition that such a pointed plane $(H,p)$ meets a hypersurface at the point $p$ to a prescribed order. Such a bundle is known as the \emph{bundle of principal parts} \cite[D\'efinition (16.3.1)]{MR238860}.
\begin{defn}\label{defn:pp_forPhi}
    Let $m$ be a non-negative integer, and $X\rightarrow Y$ a morphism of smooth schemes. For $i=1,2$, let $p_i\colon X\times_YX\to X$ denote projection onto the $i^\text{th}$ factor. Let $\mc{I}_\Delta$ denote the ideal sheaf of the diagonal embedding $\Delta\colon X\hookrightarrow X\times_YX$. The \emph{$m^\text{th}$ bundle of relative principal parts} $\mc{P}^m(\mathcal{L})\to X$ with respect to a line bundle $\mathcal{L}\rightarrow X$   is defined to be
    \[\mc{P}_{X/Y}^m(\mathcal{L}):=p_{2*}(p_1^*\mathcal{L}\otimes\mc{O}_{X\times_YX}/\mc{I}^{m+1}_\Delta).\]
\end{defn}
\begin{notn}
    Throughout this section, we fix $\Phi:=\Phi_{1,n}$ and $G:=\PGr(1,n)$. To simplify our indexing, we will use the notation $\mc{E}_{m+1}:=\mc{P}_{\Phi/G}^m(\mc{O}_\Phi(d))$.
\end{notn}

Recall that a vector bundle $V\to X$ is called \emph{relatively orientable} if there exists a line bundle $\mc{L}\to X$ such that $\det{V}\otimes\omega_X\cong\mc{L}\otimes\mc{L}$, where $\omega_X$ is the canonical bundle of $X$. Relative orientability is necessary for $V\to X$ to admit a well-defined Euler number valued in $\GW(k)$. It will turn out that $\mc{E}_m\to\Phi$ is not relatively orientable in our case of interest, which we now explain.

To check whether or not $\mc{E}_m\to\Phi$ is relatively orientable, we need to compute $\det\mc{E}_m$. Note that $\mc{E}_1=\mc{O}_\Phi(d)$. For any positive integer $m$, there is an exact sequence
\begin{equation}\label{eq:ses for bundle}
0\longrightarrow\mc{O}_\Phi(d)\otimes\Sym^{m-1}(\Omega_{\Phi/G})\longrightarrow\mc{E}_m\longrightarrow\mc{E}_{m-1}\longrightarrow 0.
\end{equation}
It follows (by induction) that $\rank\mc{E}_m=\binom{r+m-1}{r}$ (see e.g.~\cite[(3.5) Proposition]{MR0626480}). We can now use Equation~\ref{eq:ses for bundle} to compute $\det\mc{E}_m$.

\begin{lem}\label{lem:determinant of E_m}
    Let $m$ be a positive integer. Then
    \[\det\mc{E}_m\cong\mc{O}_\Phi\left(d\binom{r+m-1}{m-1}-(r+1)\binom{r+m-1}{m-2}\right)\otimes\pi^*\mc{O}_G\left(\binom{r+m-1}{m-2}\right),\]
    where we set $\binom{a}{-1}=0$ for $a>0$.
\end{lem}
\begin{proof}
    We argue by induction on $m$. Note that $\det\mc{E}_1=\mc{E}_1=\mc{O}_\Phi(d)$, as desired. Now suppose the result holds for some $m$. By Equation~\ref{eq:ses for bundle}, we have
    \[\det\mc{E}_{m+1}\cong\det\mc{E}_m\otimes\det(\mc{O}_\Phi(d)\otimes\Sym^m(\Omega_{\Phi/G})).\]
    Since $\Omega_{\Phi/G}$ has rank $r$, the vector bundle $\Sym^m(\Omega_{\Phi/G})$ has rank $\binom{r+m-1}{m}$. Thus
    \[\det(\mc{O}_\Phi(d)\otimes\Sym^m(\Omega_{\Phi/G}))\cong\mc{O}_\Phi\left(d\binom{r+m-1}{m}\right)\otimes\det\Sym^m(\Omega_{\Phi/G}).\]
    Finally, we have
    \[\det\Sym^m(\Omega_{\Phi/G})\cong\mc{O}_\Phi\left((-r-1)\binom{r+m-1}{m-1}\right)\otimes\pi^*\mc{O}_G\left(\binom{r+m-1}{m-1}\right)\]
    using standard computations (see e.g.~\cite[Exercise~II.5.16(c)]{MR0463157} and \cite[B.5.8]{MR1644323}). The desired result now follows from the addition rules of binomial coefficients.
\end{proof}

We can now give necessary and sufficient criteria for $\mc{E}_m\to\Phi$ to be relatively orientable.

\begin{prop}\label{prop:relatively orientable}
    The vector bundle $\mc{E}_m\to\Phi$ satisfies $\rank\mc{E}_m=\dim\Phi$ and is relatively orientable if and only if the following conditions hold:
    \begin{align*}
    r(n-r)+n &= \binom{r+m-1}{m-1}, \tag{rank condition} \\
    d\binom{r+m-1}{m-1} &\equiv (r+1)\left[\binom{r+m-1}{m-2}-1\right]\mod 2, \tag{parity of $\mc{O}_\Phi$}  \\
    \binom{r+m-1}{m-2} &\equiv n\mod 2.\tag{parity of $\pi^*\mc{O}_G$} 
    \end{align*}
\end{prop}
\begin{proof}
    Since $\dim\Phi=\dim{G}+r=r(n-r)+n$, the first condition is equivalent to $\rank\mc{E}_m=\dim\Phi$.
    
    The Picard group of $\Phi$ is generated by $\pi^*\mc{O}_G(1)$ and $\mc{O}_\Phi(1)$ (see e.g.~\cite[Exercise~II.7.9(a)]{MR0463157}), so $\mc{E}_m$ is relatively orientable if and only if
    \[\det\mc{E}_m\otimes\omega_\Phi=\mc{O}_\Phi(A)\otimes\pi^*\mc{O}_G(B)\]
    for some even numbers $A$ and $B$. Since 
    \[\omega_\Phi=\mc{O}_\Phi(-r-1)\otimes\pi^*\mc{O}_G(-n)\]
    and $\det\mc{E}_m$ is given in Lemma~\ref{lem:determinant of E_m}, we have the result.
\end{proof}

\begin{rem}
One can check that $\mc{E}_m$ is not orientable if $r=1$, as pointed out in \cite[Remark~6.4]{osculating}. Using the code of \cite{GabRep}, we found that for $r, m, n\le 5000$ and $r<n$, the conditions of Proposition~\ref{prop:relatively orientable} are only satisfied by tuples $(d, r, m, n)$ with $d$ even and $(r, m, n)$ one of the following:
\[(3, 5, 11),\qquad(3, 21, 445),\qquad(3, 37, 2287).\]
\end{rem}

While we are only interested in the $r=1$ case for this article, we can say something about the Euler number when $\mc{E}_m\to\Phi$ is relatively orientable.

\begin{prop}
    If $\mc{E}_m\to\Phi$ is relatively orientable with $\rank\mc{E}_m=\dim\Phi$, then its Euler number is given by $\frac{c_\mr{top}(\mc{E}_m)}{2}\mb{H}$, where $c_\mr{top}$ denotes the top Chern number over $\mb{C}$.
\end{prop}
\begin{proof}
    Euler numbers of vector bundles over odd dimensional schemes are hyperbolic \cite[Proposition 19]{FourLines}, so $e(\mc{E}_m)$ is some multiple of $\langle 1\rangle+\langle -1\rangle$ whenever $\dim\Phi$ is odd. Since $\rank e(\mc{E}_m)$ is equal to the complex count (i.e.~the top Chern number of $\mc{E}_m$), the desired formula will follow if we can show that $\dim\Phi$ is odd in the relatively orientable case. We will show the contrapositive. 
    
    Let $B=\binom{r+m-1}{m-2}$. If $\dim\Phi=r(n-r)+n$ is even, then $n$ and $r$ must both be even. Thus $d(r(n-r)+n)$ must be even, so $(r+1)(B-1)$ must be even as well. It follows that $B$ must be odd. But this contradicts $B\equiv n\mod 2$, so $\mc{E}_m\to\Phi$ cannot be relatively orientable.
\end{proof}

\subsection{Orienting divisors}\label{sec:orienting divisors}
As we have seen, the bundle $\mc{E}_m\to\Phi$ is often not relatively orientable, notably in our cases of interest (i.e.~when $r=1$). In order to resolve this issue, we will follow the ideas of \cite{LV21} and work relative to a divisor.

\begin{defn}\label{defn:orientable_relative}
    Let $X\to\Spec{k}$ be a smooth, proper $k$-scheme. A vector bundle $V$ over $X$ is said to be \emph{relatively orientable relative to a divisor} $D\subset X$ if there exists a line bundle $\mc{L}$ and an isomorphism  $\rho\colon \det V\otimes\omega_{X/k}\otimes\mc{O}(D)\to\mc{L}^{\otimes 2}$. We will refer to such a divisor as an \emph{orienting divisor} for $V\to X$.
\end{defn}

To find a suitable orienting divisor for $\mc{E}_m\to\Phi$, we focus on our case of interest. That is, set $r=1$ and assume that $\rank\mc{E}_m=\dim\Phi$. This gives us an equality $2n-1=\binom{m}{m-1}$, so we find that $m=2n-1$. Under these assumptions, Lemma~\ref{lem:determinant of E_m} implies that
\[\det\mc{E}_{2n-1}\otimes\omega_\Phi\cong\mc{O}_\Phi(d(2n-1)-4n^2+6n-4)\otimes\pi^*\mc{O}_{G}(2n^2-4n+1).\]

For $D\subset\Phi$ to be an orienting divisor, it therefore suffices to require $\mc{O}_\Phi(D)\cong\pi^*\mc{O}_{G}(1)$ (when $d$ is even) or $\mc{O}_\Phi(D)\cong\mc{O}_\Phi(1)\otimes\pi^*\mc{O}_{G}(1)$ (when $d$ is odd). When $d$ is even, we may thus take $D\subset\Phi$ to be a divisor parameterizing pointed lines meeting a fixed codimension 2 linear subspace of $\mb{P}^n$. When $d$ is odd, we may take $D$ to be a divisor parameterizing pointed lines that either meet a fixed codimension 2 linear subspace or whose marked point lies on a fixed codimension 1 hyperplane. We will give explicit descriptions of such divisors in the next subsection.

\begin{rem}
We conclude this subsection by remarking that while $\mc{E}_m$ admits a well-defined Euler number (valued in $\GW(k)$) relative to an orienting divisor $D$, this Euler number will depend on $D$. By restricting our attention to $k=\mb{R}$, we could investigate how this Euler number behaves with respect to the configuration of $D(\mb{R})$ and the real points of our hypersurface in question, analogous to \cite{LV21}. However, we will not pursue this question in this article.
\end{rem}

\subsection{Trivializations}
As a final step before proving our main results, we need to choose local trivializations of $\mc{E}_m$ and $\mc{O}_\Phi(D)$ for an orienting divisor $D$. We also need to verify that our chosen trivializations are compatible with the coordinates we gave in Section~\ref{sec:coordinates}. Throughout this subsection, we fix an open affine $U_{I,\ell}\subset\Phi$ over which to work.

Our preferred coordinates on $\Phi$ give us a natural choice of trivialization for $\mc{O}_\Phi(D)$. The proof of the following lemma is standard (c.f.~\cite[Lemma~3.8 and Proposition~5.1]{osculating}), so we opt to omit it.

\begin{lem}\label{lem:local trivs for 3 bundles}
    The line bundle $\pi^*\mc{O}_G(1)$ is locally trivialized over $U_{I,\ell}$ by
    \[\frac{w_{i_1,i_2}}{w_{1,2}}:=
    \frac{{\tilde x_{1,i_1}\tilde x_{2,i_2}-\tilde x_{1,i_2}\tilde x_{2,i_1}}}{{\tilde x_{1,1}\tilde x_{2,2}-\tilde x_{1,2}\tilde x_{2,1}}}.\] 
    The line bundle $\mc{O}_\Phi(1)$ is locally trivialized over $U_{I,\ell}$ by
    \[\frac{z_{i_\ell}}{z_{1}}:=
    \frac{y_{1}\tilde x_{1,i_\ell}+y_{2}\tilde x_{2,i_\ell}}{y_{1}\tilde x_{1,1}+y_{2}\tilde x_{2,1}}.\]
    The line bundle $\det\mc{E}_m$ is locally trivialized on $U_{I,\ell}$ by
    \[s_{I,\ell}:=\left(\frac{z_{i_\ell}}{z_{1}}\right)^N
    \left(\frac{w_{i_1,i_2}}{w_{1,2}}\right)^M,\]
    where $N=m(d-m+1)$ and $M=\frac{m(m-1)}{2}$.
\end{lem}

For our orienting divisor, we may therefore take $D=\mb{V}(w_{1,2})$ (when $d$ is even) or $D=\mb{V}(w_{1,2}\cdot z_1)$ (when $d$ is odd). The former divisor parameterizes pointed lines meeting a particular codimension 2 linear subspace $H\subset\mb{P}^n$, while the latter divisor parameterizes pointed lines that meet $H$ or whose marked point lies on a particular codimension 1 hyperplane.

For our local trivialization of $\mc{E}_m$, we again work over an open patch $U_{I,\ell}\subset\Phi$. Assume $r=1$. The short exact sequence in Equation~\ref{eq:ses for bundle} implies that $\mc{E}_m$ has a filtration by successive quotients
\begin{equation}\label{eq:filtration}
\mc{O}_\Phi(d),\mc{O}_\Phi(d)\otimes\Omega_{\Phi/G},\ldots,\mc{O}_\Phi(d)\otimes\Sym^{m-1}(\Omega_{\Phi/G}).
\end{equation}
Since $r=1$, we have that $\mc{O}_\Phi(d)\otimes\Sym^a(\Omega_{\Phi/G})$ is a line bundle for $0\leq a\leq m-1$. We can obtain a local trivialization
\[\tau_{I,\ell}\colon \mc{E}_m|_{U_{I,\ell}}\longrightarrow\mb{A}^m\]
by locally trivializing the line bundles $\mc{O}_\Phi(d)\otimes\Sym^a(\Omega_{\Phi/G})$ for $0\leq a\leq m-1$. Recall that for any finite rank vector bundle $V$, we have a canonical isomorphism
\begin{equation}\label{eq:sym=gamma}
\Sym^a(V)\cong\Gamma^a(V^\vee)^\vee,
\end{equation}
where $\Gamma^a$ denotes $a\textsuperscript{th}$ divided powers and $(-)^\vee$ denotes the dual. In particular, we have $\Sym^a(\Omega_{\Phi/G})\cong\Gamma^a(\T_{\Phi/G})^\vee$, where $\T_{\Phi/G}$ is the relative tangent bundle. As 
\[t_{I,\ell}:=y_1\tilde{x}_{1,i_{\ell'}}+y_2\tilde{x}_{2,i_{\ell'}}\]
is our local coordinate of the fiber $\mb{P}^1\to\Phi\to G$, the bundle $\T_{\Phi/G}$ is locally trivialized by
\[\partial_{I,\ell}:=\frac{\partial}{\partial t_{I,\ell}}.\]
Moreover, $t_{I,\ell}(L,p)$ vanishes at $p$ for each $(L,p)$, so $t_{I,\ell}$ is a local parameter of $\mc{O}_{L,p}$. By Equation~\ref{eq:sym=gamma}, $\Sym^a(\Omega_{\Phi/G})$ is locally trivialized by $\Gamma^a(\partial_{I,\ell})^\vee$. The divided power differential $\Gamma^a(\partial_{I,\ell})$ is known as the \emph{Hasse derivative}.

\begin{defn}
    The \emph{$a\textsuperscript{th}$ Hasse derivative} is the generalized derivation
    \[D_t^{(a)}\colon k[t]\to k[t]\]
    that is determined by the rules
    \[D_t^{(a)}t^n=\begin{cases} \binom{n}{a}t^{n-a} & n\geq a,\\
    0 & \text{otherwise}.
    \end{cases}\]
\end{defn}

A key feature of Hasse derivatives is that they satisfy a form of Taylor's theorem\footnote{The Hasse derivative can also be defined for partial differentiation with several variables by multi-indexing. Taylor's theorem holds in this greater generality, but we will not need this for our purposes.}:

\begin{prop}\label{prop:hasse-taylor}
    Let $X$ be a $k$-scheme of dimension 1 with regular closed point $p$ (i.e. $\mc{O}_{X,p}$ is a regular local ring). Let $t$ be a local parameter of $\mc{O}_{X,p}$. If $f\in\mc{O}_{X,p}$, then
    \[f=\sum_{a=0}^\infty D_t^{(a)}f(p)\cdot t^a.\]
\end{prop}
\begin{proof}
    A proof can be found in \cite[Corollary~2.5.14]{Gol03}.
\end{proof}

We can now define our local trivializations of $\mc{E}_m$, which will be given by taking the first $m$ jets of a section with respect to the Hasse derivative.

\begin{lem}\label{lem:local trivs}
    Let $\tau_{I,\ell}\colon \mc{E}_m|_{U_{I,\ell}}\to\mb{A}^m$ be defined by 
    \[\tau_{I,\ell}(F,(L,p))=(F(p),D^{(1)}_{t_{I,\ell}}F(p),\ldots,D^{(m-1)}_{t_{I,\ell}}F(p))\]
    for each $F\in\mc{O}_{\mb{P}^n}(d)$. Then $\tau_{I,\ell}$ is a local trivialization of $\mc{E}_m$.
\end{lem}
\begin{proof}
    First, note that if $(L,p)\in U_{I,\ell}$, then $t_{I,\ell}$ is a local parameter of $L$ at $p$. Moreover, $p$ is a regular closed point of $L$. Since $\mc{O}_\Phi(d)\cong\beta^*\mc{O}_{\mb{P}^n}(d)$, the desired result follows from the filtration given in Equation~\ref{eq:filtration}, Taylor's theorem for Hasse derivatives (Proposition~\ref{prop:hasse-taylor}), and our observation that $\Gamma^a(\partial_{I,\ell})^\vee$ locally trivializes $\Sym^a(\Omega_{\Phi/G})$.
\end{proof}

To conclude, let $m=2n-1$ (so that $\rank\mc{E}_m=\dim\Phi_{1,n}$). We need to show that our local coordinates and local trivializations are compatible with our relative orientation of $\mc{E}_{2n-1}$ relative to $D$. This will immediately follow from Lemma~\ref{lem:local trivs for 3 bundles} once we prove the following lemma.

\begin{lem}
    Under the isomorphism $\det\mc{E}_m\cong\mc{O}_\Phi(m(d-m+1))\otimes\pi^*\mc{O}_G(\frac{m(m-1)}{2})$, the trivialization $\det\tau_{I,\ell}$ is sent to $\left(\frac{z_{i_\ell}}{z_1}\right)^{m(d-m+1)}\left(\frac{w_{i_1,i_2}}{w_{1,2}}\right)^{m(m-1)/2}$.
\end{lem}
\begin{proof}
    The isomorphism $\det\mc{E}_m\cong\mc{O}_\Phi(m(d-m+1))\otimes\pi^*\mc{O}_G(\frac{m(m-1)}{2})$ was constructed inductively by the isomorphisms
    \[\det\mc{E}_{a+1}\cong\det\mc{E}_a\otimes\det(\mc{O}_\Phi(d)\otimes\Sym^a(\Omega_{\Phi/G})),\]
    and our trivializations $\tau_{I,\ell}$ were defined by trivializing each $\mc{O}_\Phi(d)\otimes\Sym^a(\Omega_{\Phi/G})$. For each $a$, the section
    \[s_a:=\left(\frac{z_{i_\ell}}{z_1}\right)^d\otimes\left(\frac{w_{i_1,i_2}}{w_{1,2}}\right)^a\]
    locally trivializes $\mc{O}_\Phi(d)\otimes\Sym^a(\Omega_{\Phi/G})$. The section $(\frac{z_{i_\ell}}{z_1})^{-2a}(\frac{w_{i_1,i_2}}{w_{1,2}})^a$ is the image of the trivialization $\Gamma^a(\partial_{I,\ell})^\vee$ of $\Sym^a(\Omega_{\Phi/G})$, and thus $s_a$ is the image of the section $\sigma_a(F,(L,p))=D^{(a)}_{t_{I,\ell}}F(p)$. It follows that
    \[\bigotimes_{a=0}^{m-1}\sigma_a=\left(\frac{z_{i_\ell}}{z_1}\right)^{m(d-m+1)}\left(\frac{w_{i_1,i_2}}{w_{1,2}}\right)^{m(m-1)/2},\]
    as desired.
\end{proof}

\begin{rem}
    Over the field of complex numbers, the Euler number of $\mc{E}_m\to\Phi$ is well-known. See for example \cite{bertone}, and \cite[Equation~(3.4)]{mu} for an approach using Gromov--Witten invariants.
\end{rem}

\section{Wronskian interpretation of the local index}\label{sec:wronskian}
In this section, we give a more precise statement of Theorem~\ref{thm:wronskian}, which we then prove. Throughout, we assume $r=1$ and $m=2n-1$ (as in Section~\ref{sec:orienting divisors}). Let $F$ be a global section of $\mc{O}_{\mb{P}^n}(d)$. Note that $\beta^*\mc{O}_{\mb{P}^n}(d)\cong\mc{O}_\Phi(d)$, where $\beta\colon \Phi\to\mb{P}^n$ is the projection map arising from Definition~\ref{def:Phi}. The degree $d$ hypersurface $\mb{V}(F)\subset\mb{P}^n$ determines a section
\[\sigma_F\colon \Phi\longrightarrow\mc{E}_{2n-1},\]
which is given on a pointed line $(L,p)\in\Phi$ by taking the principal parts of the form $F$ along $L\subset\mb{P}^n$. In order to compute the local index of this section, we will write out a formula for the composite
\[g_{I,\ell}:=\tau_{I,\ell}|_{U_{I,\ell}}\circ\sigma_F\circ\vphi_{I,\ell}^{-1}|_{U_{I,\ell}}\colon \mb{A}^{2n-1}\longrightarrow\mb{A}^{2n-1}\]
around a zero $(L,p)\in U_{I,\ell}\subset\Phi$. Here, $(U_{I,\ell},\vphi_{I,\ell})$ are standard local coordinates around $(L,p)$ (Definition~\ref{def:coordinates on Phi}) and $\tau_{I,\ell}$ is our chosen local trivialization (Lemma~\ref{lem:local trivs}). After we have done this, we will compute the local degree $\deg_{(L,p)}(g_{I,\ell})$ and relate it to the geometry of $\mb{V}(F)$ and its highly tangent line $(L,p)$. If the highly tangent line is geometrically simple (i.e.~geometrically reduced) and $k(L,p)/k$ is separable, then 
\begin{equation}\label{eq:jacobian}
    \deg_{(L,p)}(g_{I,\ell})=\Tr_{k(L,p)/k}\langle\Jac(g_{I,\ell})|_{(L,p)}\rangle
\end{equation}
by \cite[Lemma 9]{KassWickelgrenEKLClass}.

\begin{thm}[Wronskian interpretation, precisely]
Let $F$ be a global section of $\mc{O}_{\mb{P}^n}(d)$ with $d\geq 2n-1$. Let $L$ be a line that is tangent to $\mb{V}(F)$ at $p\in L$ to order $2n-1$. Finally, assume that $(L,p)$ is geometrically reduced as a point in $\Phi$, and that $k(L,p)/k$ is separable. Then
\[\ind_{(L,p)}\sigma_F=\Tr_{k(L,p)/k}\langle\Wr_t(\nabla\beta^* F)(p)\rangle,\]
where $\Wr_t$ denotes the Hasse--Wronskian with respect to a local parameter $t$ of $L$ and $\nabla\beta^* F$ denotes the gradient of $\beta^* F$ with respect to the local coordinates representing $(L,p)$.
\end{thm}
\begin{proof}
    By \cite{BBMMO21}, we can assume that $k(L,p)=k$ by base changing and tracing back down (if necessary). If $(L,p)\in U_{I,\ell}$, then we have
    \[\tau_{I,\ell}\circ\sigma_F\circ\vphi^{-1}_{I,\ell}(L,p)=(\beta^*F(p),D^{(1)}_{t_{I,\ell}}\beta^*F(p),\ldots,D^{(2n-2)}_{t_{I,\ell}}\beta^*F(p)).\]
    By Equation~\ref{eq:jacobian}, it suffices to compute the Jacobian of $\tau_{I,\ell}\circ\sigma_F\circ\vphi^{-1}_{I,\ell}$ in terms of the local coordinates on $U_{I,\ell}$. We will treat the case where $(L,p)\in U_{(n,n+1),2}$; the case of $(L,p)\in U_{I,\ell}$ for arbitrary $I,\ell$ is completely analogous. To simplify notation, denote $t:=t_{(n,n+1),2}$.

    On $U_{(n,n+1),2}$, we have $\tilde{x}_{1,n}\tilde{x}_{2,n+1}-\tilde{x}_{1,n+1}\tilde{x}_{2,n}\neq 0$ and $y_1\tilde{x}_{1,n+1}+y_2\tilde{x}_{2,n+1}\neq 0$. Our affine coordinates on $U_{(n,n+1),2}$ are given by
    \begin{align*}
        x_{1,i}={}&\frac{\tilde{x}_{1,i}\tilde{x}_{2,n+1}-\tilde{x}_{2,i}\tilde{x}_{1,n+1}}{\tilde{x}_{1,n}\tilde{x}_{2,n+1}-\tilde{x}_{2,n}\tilde{x}_{2,n+1}},\\
        x_{2,i}={}&\frac{\tilde{x}_{2,i}\tilde{x}_{1,n}-\tilde{x}_{1,i}\tilde{x}_{2,n}}{\tilde{x}_{1,n}\tilde{x}_{2,n+1}-\tilde{x}_{2,n}\tilde{x}_{2,n+1}},\\
        y={}&\frac{y_1\tilde{x}_{1,n}+y_2\tilde{x}_{2,n}}{y_1\tilde{x}_{1,n+1}+y_2\tilde{x}_{2,n+1}}
    \end{align*}
    for $1\leq i\leq n-1$. Note that $y$ is precisely the dehomogenization of our local parameter $t$, so that $D^{(a)}_t=D^{(a)}_y$ on $U_{(n,n+1),2}$.
    
    Let $[z_1:\ldots:z_{n+1}]$ denote coordinates on $\mb{P}^n$. In our coordinates on $U_{(n,n+1),2}$, the map $\beta\colon \Phi\to\mb{P}^n$ is given by
    \[\beta(x_{1,1},x_{2,1},\ldots,x_{1,n-1},x_{2,n-1},y)=(x_{1,1}y+x_{2,1},\ldots,x_{1,n-1}y+x_{2,n-1},y),\]
    where we have written this point using the standard coordinates for $\{z_{n+1}\neq 0\}\subset\mb{P}^n$ (since $y_1\tilde{x}_{1,n+1}+y_2\tilde{x}_{2,n+1}$ corresponds to $z_{n+1}$ under $\beta$). Thus
    \begin{align*}
    \beta^*F&=F(x_{1,1}y+x_{2,1},\ldots,x_{1,n-1}y+x_{2,n-1},y,1).
    \end{align*}
    Here, the argument 1 in $F$ corresponds to the fact that we have dehomogenized by scaling all projective coordinates by $\frac{1}{z_{n+1}}$.     We denote the gradient with respect to the coordinates $(x_{1,1},x_{2,1},\ldots,x_{1,n-1},x_{2,n-1},y)$ by $\nabla$. Recall that $D^{(a)}_t=D^{(a)}_y$ on $U_{(n,n+1),2}$. One can now check that $\nabla$ and $D^{(a)}_y$ commute. For example, we have
    \begin{align*}
        \frac{\partial}{\partial y}D^{(a)}_y y^m&=\frac{\partial}{\partial y}\binom{m}{a}y^{m-a}\\
        &=\binom{m}{a}(m-a)y^{m-a-1},\\
        D^{(a)}_y\frac{\partial}{\partial y}y^m&=D^{(a)}_ymy^{m-1}\\
        &=m\binom{m-1}{a}y^{m-a-1}.
    \end{align*}
    Since $\binom{m}{a}(m-a)=m\binom{m-1}{a}$, we conclude that $\frac{\partial}{\partial y}D^{(a)}_y=D^{(a)}_y\frac{\partial}{\partial y}$. Verifying that $D^{(a)}_y$ commutes with $\frac{\partial}{\partial x_{i,j}}$ is even simpler, as $y$ and $x_{i,j}$ do not depend on one another.
    
    Altogether, we have proved that the Jacobian determinant of $(\beta^*F(p),\ldots,D^{(2n-2)}_t\beta^*F(p))$ is given by the determinant of
    \[\begin{pmatrix}
        \nabla\beta^*F(p)\\
        \nabla D^{(1)}_t\beta^*F(p)\\
        \vdots\\
        \nabla D^{(2n-2)}_t\beta^*F(p)
    \end{pmatrix}=
    \begin{pmatrix}
        \nabla\beta^*F(p)\\
        D^{(1)}_t\nabla\beta^*F(p)\\
        \vdots\\
        D^{(2n-2)}_t\nabla\beta^*F(p)
    \end{pmatrix}.\]
    The determinant of the latter matrix is precisely the Wronskian of $\nabla\beta^*F$ (with respect to the Hasse derivative) evaluated at $p$, which we denote by $\Wr_t(\nabla\beta^*F)$.
\end{proof}

\begin{rem}
    Brazelton has some results on geometrically interpreting Wronskians in enriched enumerative geometry \cite{Bra25}. The Hasse--Wronskian also arises as the local index for the enriched count of inflection points to superelliptic curves \cite{CDGLHS21}, showing that our Theorem~\ref{thm:wronskian} is compatible with this result.
\end{rem}

\section{Highly tangent lines and the second fundamental form}\label{sec:fundamental form}

In this section, we will prove Theorem~\ref{thm:fundamental form}. First, we need to formulate the problem of counting highly tangent lines to a plane curve $X$ in terms of the second fundamental form of $X$. We will begin with a more general setup, then restrict our attention to curves when necessary.

Let us suppose $X=\Spec(A)\subset \mathbb{P}^n$ is smooth affine, and let $I$ be the kernel of the product map $A\otimes_k A \rightarrow A$. For each $m\ge 0$, there is a short exact sequence of left $A$-modules
\begin{equation}
    0\longrightarrow I^m/I^{m+1}
    \longrightarrow (A\otimes_k A)/I^{m+1}
    \overset{\pi_m}{\longrightarrow}
     (A\otimes_k A)/I^{m}
    \longrightarrow 0.
\end{equation}
We call the following map of left $A$-modules the \emph{Taylor series map} (see \cite[Proposition~6.4.3]{MR1308020})
\begin{equation}\label{eq:taylor}
    \delta_m\colon A^{\oplus n+1} \longrightarrow (A\otimes_k A)/I^{m+1},
\end{equation}
given by sending $e_1=(1, 0, 0, \ldots, 0)$ to (the class of) $1\otimes a_1$, $e_2=(0, 1, 0, \ldots, 0)$ to $1\otimes a_2$, and so on. Here, $a_1, a_2,\ldots, a_{n+1}$ are the generators of $\Gamma(\mathbb{P}^n, \mc{O}_{\mathbb{P}^n}(1))$ restricted to $X$. 

When $m=0$, we have that $(A\otimes_k A)/I^{m+1}=A$, and the map \eqref{eq:taylor} is the natural evaluation map $\Gamma(\mathbb{P}^n, \mc{O}_{\mathbb{P}^n}(1))\otimes \mc{O}_{\mathbb{P}^n}(-1)\rightarrow \mc{O}_{\mathbb{P}^n}$ restricted to $X$. By the properties of the kernel, and the fact that $\pi_m \circ \delta_{m} = \delta_{m-1}$, there exists a map
\begin{equation}
    \ker(\delta_{m-1}) \longrightarrow I^m/I^{m+1}\cong \Sym^{m}\Omega_X,
\end{equation}
called \emph{$m^\mathrm{th}$ fundamental form}. We have $\ker(\delta_0)\cong \Omega_{\mathbb{P}^n}|_X$ and $\ker(\delta_1)\cong N^\vee_{X/\mb{P}^n}$ (see \cite[Theorem II.8.13 and Proposition II.8.12]{MR0463157}).  If $X$ is a smooth projective variety, those maps naturally glue together on the affine patches of $X$. We are interested in one particular fundamental form.
\begin{defn}
    Let $X\subset\mathbb{P}^n$ be a smooth projective variety. The map 
    \[\II\colon N^\vee_{X/\mb{P}^n}\longrightarrow\Sym^2\Omega_X\]
    is called the \emph{second fundamental form} of $X\subset\mb{P}^n$. By hom-tensor adjunction, we can think of $\II$ as a global section of the bundle 
    \[\mc{F}:=\Hom(N^\vee_{X/\mb{P}^n},\Sym^2\Omega_X)\cong N_{X/\mb{P}^n}\otimes\Sym^2\Omega_X.\]
\end{defn}
\begin{rem}
    The second fundamental form can be constructed by applying \cite[Section I.3]{MR274461} to the exact sequence (e.g.~\cite[Theorem II.8.17 (2)]{MR0463157}):
    \begin{equation}\label{eq:cotangent_seq}
    0\longrightarrow N^\vee_{X/\mb{P}^n}\longrightarrow\Omega_{\mb{P}^n}|_X\longrightarrow\Omega_X\longrightarrow 0.
    \end{equation}
    In \cite{einniu,MP24}, the $m^\mathrm{th}$ fundamental forms are constructed using the Taylor series map applied to the bundles of principal part $\mathcal{P}_{X/k}^m(\mathcal{O}_X)$ \cite[D\'efinition (16.3.6)]{MR238860}, and using again the properties of the kernel. They suppose $k$ to be algebraically closed, but this hypothesis is not really necessary, at least in order to define the fundamental forms\footnote{Confirmed by Wenbo Niu by a personal communication.}. Our approach follows \cite[Appendix A]{MR1308020}.
\end{rem}

We now assume $X\subseteq\mb{P}^2$ is a curve and characterize when $\mc{F}\to X$ is relatively orientable. 

\begin{prop}\label{prop:relatively_ori}
    Let $X\subseteq\mb{P}^2$ be a smooth projective plane curve of degree $d$. The following are equivalent.
    \begin{enumerate}
        \item $\mc{F}$ is relatively orientable.
        \item $\mc{O}_X(1)$ is a square.
        \item $X$ admits a theta characteristic and $d$ is even.
    \end{enumerate}
\end{prop}
\begin{proof}
    Recall that $N_{X/\mb{P}^2}\cong\mc{O}_{\mb{P}^2}(d)|_X=\mc{O}_X(d)$. Moreover, from the cotangent exact sequence \eqref{eq:cotangent_seq}, we have
    \[\Omega_X\cong\det(\Omega_{\mb{P}^2})|_X \otimes N_{X/\mb{P}^2}\cong\mc{O}_X(d-3),\]
    thus $$\Omega_X\otimes\mc{F}\cong\Sym^3\Omega_X\otimes N_{X/\mb{P}^2}\cong  \mc{O}_X(2d-5)^{\otimes 2}\otimes\mc{O}_X(1).$$
    This implies that $\Omega_X\otimes\mc{F}$ is a square if and only if $\mc{O}_X(1)$ is the square of a line bundle $L$. Since $d$ is the degree of $\mc{O}_X(1)$, $2\deg L = d$ implies $d$ even. In this case, we have
    $$\Omega_X\otimes\mc{F}\cong\Omega_X\otimes
    \left(
    \mc{O}_X(d/2)\otimes\Omega_X
    \right)^{\otimes 2}.
    $$
    So, $\mc{F}$ is orientable if and only if $X$ admits a theta characteristic. 
\end{proof}

\begin{rem}
    In Proposition~\ref{prop:relatively_ori}, we cannot drop the hypothesis that $X$ admits a theta characteristic. For example, a non-singular real plane curve of degree $4$ with no real points has no theta characteristics \cite[Section~5]{MR286136}. See also \cite[Section~5]{MR631748} for a thorough discussion about theta characteristics of real curves, and \cite{MR2672429} for a characterization of the existence of theta characteristic in terms of the splitting of the stable motivic homotopy type of the curve. 
    If $k$ is algebraically closed, then any non-singular curve admits theta characteristics \cite[Section~4]{MR292836}.
\end{rem}

An advantage of working with $\mc{F}$ instead of $\mc{E}_m$ (from earlier in the paper) is that $\mc{F}\to X$ is relatively orientable in many cases where $\mc{E}_m\to\Phi$ is not. In order to extract enumerative information from $\mc{F}$, we have a few more tasks at hand. First, we need to check that the vanishing locus of $\II$ corresponds to the locus of points on $X$ at which there is a highly tangent line. Second, we need to compute the Euler number $e(\mc{F})\in\GW(k)$. Finally, we need to compute the local index of $\II$ at inflectional points.

\begin{defn}
    Let $X\subseteq\mb{P}^2$ be a smooth projective plane curve. Given $Y,Z\subseteq\mb{P}^2$ and a point $p\in Y\cap Z$, let $i_p(Y,Z)$ denote the intersection multiplicity of $Y$ and $Z$ at $p$. Let
    \[\inf(X)=\{x\in X:\text{there is a line }L\subseteq\mb{P}^2\text{ such that }i_x(X,L)\geq 3\}\]
    denote the locus of \emph{inflectional points} on $X$. Note that this is a geometric definition: $\inf(X)$ is defined over $k$ because $X$ is defined over $k$, but individual points of $\inf(X)$ may be defined over non-trivial extensions of $k$.
\end{defn}

By construction, the second fundamental form $\II$ is a linear system of bilinear forms on the tangent vector spaces $(T_{X,p})_{p\in X}$ parameterized by the conormal bundle $N^\vee_{X/\mb{P}^2}$. In complex differential geometry, the properties of the second fundamental form are exploited in the influential paper \cite{MR1273267}. As explained on \cite[p.~367]{MR559347}, for $k=\mathbb{C}$ the second fundamental form is degenerate at a point of $T_{X,p}$ if and only if that point spans a line of $\mb{P}^2$ meeting $X$ at $p$ with order at least 3. In other words, the zero locus of the global section $\II\colon X\to\Hom(N^\vee_{X/\mb{P}^2},\Sym^2\Omega_X)$ is precisely $\inf(X)$.  

For arbitrary fields note that the functor $\otimes_k \bar k$ is exact, so we may suppose $k$ algebraically closed. Indeed, Serre's intersection formula is preserved, see \cite[Appendix A]{MR0463157}. A geometric interpretation of the degenerate locus of $\II$ as the locus of inflectional points is given by \cite[Section~3]{MP24} (see also \cite[Appendix~B,~Theorem 2.3]{MR1308020} and \cite{MR506323}).

It is now straightforward to calculate the Euler number of $\mc{F}$ in the relatively orientable case.

\begin{lem}
    If $d$ is even and $X$ admits a theta characteristic, then $e(\mc{F})=\frac{3d(d-2)}{2}\mb{H}$.
\end{lem}
\begin{proof}
    Since $\mc{F}\to X$ is a line bundle over a curve, its Euler number is hyperbolic by \cite[Proposition 19]{FourLines}. It thus suffices to compute the rank of $e(\mc{F})$, which is given by the integral Euler number over an algebraically closed field. The claim now follows from the classical count of $3d(d-2)$ inflectional points to a smooth complex plane curve of degree $d$ \cite[Exercise IV.2.3]{MR0463157}.
\end{proof}

To conclude, we prove Theorem~\ref{thm:fundamental form} by computing the local index of $\II$. This will be given in terms of the \emph{third fundamental form}
\[\III\colon R_3\longrightarrow\Sym^3\Omega_X,\]
where $R_3$ is the sheaf associated to the module $\ker(\delta_2)$.

\begin{proof}[Proof of Theorem~\ref{thm:fundamental form}]
    The latter equation follows from the first by the Poincar\'e--Hopf Theorem (see e.g.~\cite[Theorem~1.1]{BW23}) and the fact that $\langle 3\rangle\mb{H}=\mb{H}$, so it remains to treat the first equation. 
    Using the same notation at the beginning of this section, assume $X=\Spec(A)$ is localized at a smooth $k$-point. The ideal $I$ is free of rank one, so it is generated by $\d z:=z\otimes 1-1\otimes z$ for some $z\in A$. There exists an isomorphism
    \[(A\otimes_k A)/I^{3}\cong A^3\]
    sending $(a\otimes b)$ to $a(b, \partial_z b, \partial_z^2 b, \partial_z^3 b)$, where $\partial_zb$ is defined as that element of $A$ such that
    \[(b\otimes 1 - 1\otimes b)^m= \partial_z^mb(\d z)^m.\]
    In our case, $\partial_z^mb=D_z^{(m)}b$ are the Hasse derivatives of the regular element $b$ with respect to the local parameter $z$ (see \cite[Note~6.4.2]{MR1308020}). Thus, in this particular open subset, the second and third fundamental forms are given by $(\partial_z^2b) \d z\otimes \d z$ and $(\partial_z^3b)\d z\otimes\d z\otimes\d z$, respectively. Since $\frac{\partial}{\partial z}D^{(n)}_z=(n+1)D^{(n+1)}_z$, the Jacobian of $\II$ on this open subset is given by
    \[\frac{\partial}{\partial z}((\partial^2_z b)\d z\otimes\d z)=3(\partial^3_z b)\d z\otimes\d z\otimes\d z.\]
    In particular, we have $\Jac(\II)|_p=3\cdot\III(p)$.
    
    Recall that the zero locus of $\II$ is the locus of inflectional points $\inf(X)$. Because $X$ is a smooth general curve, $\inf(X)$ is zero dimensional and reduced. This implies that
    \begin{equation}\label{eq:II}
    \ind_p\II=\Tr_{k(p)/k}\langle\Jac(\II)|_{p}\rangle
    \end{equation}
    (see e.g.~\cite[p.~692]{EnrichedCubicSurface}). By base changing to $k(p)$, working $k(p)$-rationally, and then applying the field trace, we may assume (without loss of generality) that all zeros of $\II$ are $k$-rational. 
    
    Equation~\eqref{eq:II} only makes sense once we have chosen a Nisnevich coordinate on $X$ and a compatible local trivialization of $\mc{F}$. We will explain why our prior argument is independent of our choice of coordinate and trivialization. By \cite[Proposition~20]{EnrichedCubicSurface}, we always have Nisnevich coordinates $(U,\vphi)$ about $p$, where $\vphi\colon U\to\mb{A}^1$. By intersecting with our affine neighborhood $\Spec(A)$, we may assume that $U=\Spec(A)$. The map $\vphi$ corresponds to a local parameter $t$ of $\Spec(A)$, and $\d{t}:=t\otimes 1-1\otimes t$ will generate the ideal $I$. We may therefore take $z$ to be the local parameter determined by our Nisnevich coordinate. Finally, one can always obtain a trivialization of $\mc{F}$ compatible with $(U,\vphi)$ by \cite[Proposition~5.5]{McK22}. In the present case, this can be done by picking a non-vanishing local section of the line bundle $\mc{F}$ and rescaling if necessary.
\end{proof}

\begin{rem}
Higher fundamental forms for algebraic varieties made their first appearance in work of Griffiths and Harris \cite{MR559347}, and were later generalized by several authors \cite{Teg92,landsbergseoul,einniu,MP24}. It would be interesting to try extending the results of this section to hypersurfaces in higher-dimensional projective spaces. One of the first hurdles to such a generalization is that the sheaves used to define higher fundamental forms need not be vector bundles.
\end{rem}

\bibliography{highly-tangent}
\bibliographystyle{amsalpha}

\end{document}